\documentclass[12pt]{amsart}

\usepackage{framed}
\usepackage{braket}
\usepackage{amsmath,amssymb,amsthm}
\usepackage{color}
\usepackage{bm}
\usepackage[all]{xy}
\usepackage{amscd}
\usepackage{graphicx}
\usepackage{ascmac}
\usepackage{BOONDOX-frak, mathrsfs}
\usepackage[top=30truemm,bottom=30truemm,left=25truemm,right=25truemm]{geometry}
\usepackage{mathtools}
\mathtoolsset{showonlyrefs=true}
\usepackage{tikz}
\usetikzlibrary{cd}

\makeatletter
\@addtoreset{equation}{section}

\makeatother

\newcommand{\Z}{\mathbb{Z}}

\newcommand{\Q}{\mathbb{Q}}

\newcommand{\cC}{\mathcal{C}}

\newcommand{\cF}{\mathcal{F}}
\newcommand{\cG}{\mathcal{G}}
\newcommand{\cH}{\mathcal{H}}

\newcommand{\cL}{\mathcal{L}}
\newcommand{\cM}{\mathcal{M}}
\newcommand{\cN}{\mathcal{N}}
\newcommand{\cO}{\mathcal{O}}

\newcommand{\wtil}[1]{\widetilde{#1}}

\DeclareMathOperator{\Gal}{Gal}

\DeclareMathOperator{\Imag}{Im}

\DeclareMathOperator{\pd}{pd}

\DeclareMathOperator{\ram}{ram}

\DeclareMathOperator{\ab}{ab}

\DeclareMathOperator{\Hom}{Hom}
\DeclareMathOperator{\Ext}{Ext}

\DeclareMathOperator{\rank}{rank}

\DeclareMathOperator{\cyc}{cyc}

\usepackage[OT2,T1]{fontenc}
\DeclareSymbolFont{cyrletters}{OT2}{wncyr}{m}{n}
\DeclareMathSymbol{\Sha}{\mathalpha}{cyrletters}{"58}

\let\oldenumerate\enumerate
\renewcommand{\enumerate}{
   \oldenumerate
   \setlength{\itemsep}{1pt}
   \setlength{\parskip}{0pt}
   \setlength{\parsep}{0pt}
}
\let\olditemize\itemize
\renewcommand{\itemize}{
   \olditemize
   \setlength{\itemsep}{1pt}
   \setlength{\parskip}{0pt}
   \setlength{\parsep}{0pt}
}

\theoremstyle{plain}
\newtheorem{thm}{Theorem}[section]
\newtheorem{lem}[thm]{Lemma}

\newtheorem{prop}[thm]{Proposition}
\newtheorem{cor}[thm]{Corollary}
\newtheorem{claim}[thm]{Claim}

\theoremstyle{definition}
\newtheorem{defn}[thm]{Definition}

\newtheorem{rem}[thm]{Remark}
\newtheorem{eg}[thm]{Example}

\newcommand{\RG}{\mathsf{R}\Gamma}

\DeclareMathOperator{\Iw}{Iw}
\DeclareMathOperator{\cd}{cd}

\makeatletter
\@namedef{subjclassname@2020}{\textup{2020} Mathematics Subject Classification}
\makeatother

\title
[Pseudo-null Iwasawa modules]
{Rarity of pseudo-null Iwasawa modules for $p$-adic Lie extensions}
\author[T.~Kataoka]{Takenori Kataoka}
\address{Department of Mathematics, Faculty of Science Division II, Tokyo University of Science.
1-3 Kagurazaka, Shinjuku-ku, Tokyo 162-8601, Japan}
\email{tkataoka@rs.tus.ac.jp}
\keywords{Iwasawa theory, Iwasawa modules, Greenberg's generalized conjecture}
\subjclass[2020]{11R23}
\date{\today}

\begin{document}

\maketitle

\begin{abstract}
In this paper, we obtain a necessary and sufficient condition for the pseudo-nullity of the $p$-ramified Iwasawa module for $p$-adic Lie extension of totally real fields.
It is applied to answer the corresponding question for the minus component of the unramified Iwasawa module for CM-fields.
The results show that the pseudo-nullity is very rare.
\end{abstract}

\section{Introduction}\label{sec:intro}

One of the main themes in Iwasawa theory is the study of various Iwasawa modules, including unramified Iwasawa modules.
It is a classical conjecture (see Greenberg \cite[Conjecture 3.5]{Gree01}) that the unramified Iwasawa modules for maximal multiple $\Z_p$-extensions are always pseudo-null.
This conjecture is often called Greenberg's generalized conjecture and a lot of research has been done, but it is still open.

It is a remarkable result of Hachimori--Sharifi \cite{HS05} that, when we consider non-commutative $p$-adic Lie extensions, the pseudo-nullity (in the sense of Venjakob \cite{Ven02}) of unramified Iwasawa modules often fails.
More concretely, assuming that the extension is ``strongly admissible'' and $\mu = 0$, they obtained a necessary and sufficient condition for the pseudo-nullity of the minus components (see Theorem \ref{thm:HS05}).

One of the main result of this paper (Theorem \ref{thm:main_CM}) claims a complete necessary and sufficient condition for the pseudo-nullity of the minus components of unramified Iwasawa modules, {\it without} assuming that the extension is ``strongly admissible'' or $\mu = 0$.
In fact, our first main theorem (Theorem \ref{thm:main}) handles $p$-ramified Iwasawa modules for totally real fields.
We then deduce Theorem \ref{thm:main_CM} from Theorem \ref{thm:main} and its variant by using a suitable duality.

\subsection{Main theorem for totally real fields}\label{ss:intro_real}

Let $p$ be an odd prime number throughout this paper.
Let $F$ be a totally real number field and $L/F$ a pro-$p$, $p$-adic Lie extension (so $L$ is also totally real).
We suppose the following:
\begin{itemize}
\item
$L \supset F^{\cyc}$, where $F^{\cyc}$ denotes the cyclotomic $\Z_p$-extension of $F$.
\item
$L/F$ is unramified at almost all primes of $F$.
\end{itemize}
Set $\cG = \Gal(L/F)$ and $\cH = \Gal(L/F^{\cyc})$, so $\cG/\cH \simeq \Gal(F^{\cyc}/F) \simeq \Z_p$.
We are mainly interested in the non-commutative case:
If $\cG$ is commutative, Leopoldt's conjecture predicts that $\cH$ is finite.

Let $S_p = S_p(F)$ be the set of $p$-adic primes of $F$.
Let $S \supset S_p$ be a finite set of finite primes of $F$.
We study the $S$-ramified Iwasawa module $X_S(L) = \Gal(M_S(L)/L)$, where $M_S(L)$ denotes the maximal abelian $p$-extension of $L$ unramified outside $S$.
It is known that $X_S(L)$ is finitely generated over the Iwasawa algebra $\Z_p[[\cG]]$.
Note that we {\it do not} assume that $L/F$ is unramified outside $S$.
The case where $S = S_p$ will be applied to deduce Theorem \ref{thm:main_CM}.

To state the result, we introduce the following:
\begin{itemize}
\item
Let $S_{\ram}^S = S_{\ram}^S(L/F^{\cyc})$ be the set of primes of $F^{\cyc}$ that are ramified in $L/F^{\cyc}$ and not lying above $S$.
\item
Let $\lambda_S$, $\mu_S$ be the Iwasawa $\lambda$, $\mu$-invariants of $X_S(F^{\cyc})$.
\end{itemize}
In fact, we have $\mu_S = \mu_{S_p}$ and $\lambda_S$ can be described by using $\lambda_{S_p}$ (see Lemma \ref{lem:S_varies}).

\begin{thm}\label{thm:main}
The $\Z_p[[\cG]]$-module $X_S(L)$ is pseudo-null if and only if $\mu_S = 0$ and (exactly) one of (i) and (ii) holds:
\begin{itemize}
\item[(i)]
$\dim \cG = 1$, $\lambda_S = 0$, and $\# S_{\ram}^S \leq 1$.
\item[(ii)]
$\dim \cG = 2$ and $\lambda_S + \# S_{\ram}^S = 1$.
\end{itemize}
\end{thm}

This theorem shows that the pseudo-nullity of $X_S(L)$ is very rare.
See Example \ref{eg:Q} for a description for the case $F = \Q$.
In general, necessary conditions include $\mu_S = 0$, $\lambda_S + \# S_{\ram}^S \leq 1$, and $\cH$ is pro-cyclic (see Lemma \ref{lem:H_cyclic}(2)).
We note that Lemma \ref{lem:H_cyclic}(2) also shows that the condition $\dim \cG = 2$ in (ii) can be replaced by $\dim \cG \geq 2$.

\subsection{Main theorem for CM-fields}\label{ss:intro_CM}

Let $\cF$ be a CM number field and $\cL/\cF$ be a pro-$p$, $p$-adic Lie extension that is also a CM-field.
We write $L = \cL^+$ and $F = \cF^+$ for the maximal totally real subfields.
We suppose the same assumptions in \S \ref{ss:intro_real} hold for $L/F$, that is:
\begin{itemize}
\item
$\cL \supset \cF^{\cyc}$, where $\cF^{\cyc}$ denotes the cyclotomic $\Z_p$-extension of $\cF$.
\item
$\cL/\cF$ is unramified at almost all primes of $\cF$.
\end{itemize}
Set $\cG = \Gal(\cL/\cF) \simeq \Gal(L/F)$ and $\cH = \Gal(\cL/\cF^{\cyc}) \simeq \Gal(L/F^{\cyc})$.

Let $X(\cL)$ be the unramified Iwasawa module for $\cL$.
We study its minus component $X(\cL)^-$ with respect to the complex conjugation, which is finitely generated over $\Z_p[[\cG]]$.

To state the result, we introduce the following:
\begin{itemize}
\item
Put $\delta = 1$ if $\cF$ contains $\mu_p$ and $\delta= 0$ otherwise.
Here, $\mu_p$ denotes the group of $p$-th roots of unity.
\item
Let $S_{\ram}^- = S_{\ram}^-(\cL/F^{\cyc})$ be the set of non-$p$-adic primes of $F^{\cyc}$ that are ramified in $L/F^{\cyc}$ and split in (the quadratic extension) $\cF^{\cyc}/F^{\cyc}$.
\item
Let $\lambda^-$, $\mu^-$ be the Iwasawa $\lambda$, $\mu$-invariants of $X(\cF^{\cyc})^-$.
\end{itemize}

\begin{thm}\label{thm:main_CM}
The $\Z_p[[\cG]]$-module $X(\cL)^-$ is pseudo-null if and only if we have $\mu^- = 0$ and (exactly) one of (i), (ii), and (iii) holds:
\begin{itemize}
\item[(i)]
$\delta = 1$, $\dim \cG = 1$, $\lambda^- = 0$, and $\# S_{\ram}^- \leq 1$.
\item[(ii)]
$\delta = 1$, $\dim \cG = 2$, and $\lambda^- + \# S_{\ram}^- = 1$.
\item[(iii)]
$\delta = 0$, $\lambda^- = 0$, and $\# S_{\ram}^- = 0$.
\end{itemize}
\end{thm}

As in Theorem \ref{thm:main}, the condition $\dim \cG = 2$ in (ii) can be replaced by $\dim \cG \geq 2$.
This theorem is stronger than the result of Hachimori--Sharifi \cite[Theorem 1.2]{HS05}, as long as we are concerned with only the pseudo-nullity (see \S \ref{ss:HS}).

\subsection{The nature of the proof}\label{ss:pf_outline}

As already remarked, we will deduce Theorem \ref{thm:main_CM} from Theorem \ref{thm:main} and its variant by using a duality.
Let us briefly discuss the proof of Theorem \ref{thm:main}.

For simplicity, for a while we assume that $L/F$ is unramified outside $S$, that is, $S_{\ram}^S = \emptyset$.
The key idea is to use the Tate sequence (Proposition \ref{prop:Tate_seq}(1)) of the form
\[
0 \to X_S(L) \to P \to Q \to \Z_p \to 0,
\]
where $P$ and $Q$ are finitely generated $\Z_p[[\cG]]$-modules of $\pd_{\Z_p[[\cG]]} \leq 1$ ($\pd$ denotes the projective dimension).
This implies that $X_S(L)$ does not have a nonzero pseudo-null submodule (Proposition \ref{prop:PN_free}).
Therefore, the pseudo-nullity of $X_S(L)$ holds only if $X_S(L) = 0$, and then the Tate sequence shows
\[
\pd_{\Z_p[[\cG]]}(\Z_p) \leq 2.
\]
This inequality holds if and only if $\cG$ is $p$-torsion-free and $\dim \cG \leq 2$ (Lemma \ref{lem:pdZ1}).
This is a severe constraint, and indeed it suffices for our purpose.

When $L/F$ is not necessarily unramified outside $S$, we still have a variant of the Tate sequence (Proposition \ref{prop:Tate_seq}(2)) that is constructed in \cite{GKK22} by Greither, Kurihara, and the author in the case where $\cG$ is commutative and $\cH$ is finite.
The construction is also valid for the general case.
The sequence involves a module denoted by $Z_{\Sigma_f \setminus S}^0(L)$, which is more complicated compared to $\Z_p$.
Even in this case, a close study of $Z_{\Sigma_f \setminus S}^0(L)$ from a homological view gives a severe constraint on $\cG$ and its decomposition groups, which suffices for our purpose.
Note that this kind of argument was done in \cite[Proposition  2.14]{GKK22} (in the much simpler case that $\cG$ is commutative and $\cH$ is finite, of course).

\begin{rem}
More generally, we can apply this idea to the Selmer groups of ordinary $p$-adic representations.
The counterpart of the Tate sequences is obtained by using suitable Selmer complexes (one may use the formalism in the author's work \cite[\S 3]{Kata_27}).
For instance, we may deal with Selmer groups of ordinary elliptic curves, which should recover the pseudo-nullity criterion obtained by Hachimori--Sharifi \cite[Theorem 5.4]{HS05}.
Another application is the $\Sigma$-ramified Iwasawa modules for $p$-ordinary CM-fields, where $\Sigma$ is a $p$-adic CM-type.
An advantage of this situation is that we naturally encounter commutative Galois groups of dimension $\geq 2$.
However, the author does not think that these kinds of Iwasawa modules are expected to be often pseudo-null.
Coates--Sujatha \cite[Conjecture B]{CS05} conjectured the pseudo-nullity for the fine Selmer groups of elliptic curves, but our formalism cannot handle the fine Selmer groups.
For this reason, we do not pursue such generalizations in this paper.
\end{rem}

\section{Preliminaries}\label{sec:alg}

In this section, we review the definition and properties of pseudo-null modules over completed group rings of compact $p$-adic Lie groups.
The concept was introduced by Venjakob \cite{Ven02} (see also \cite{Ven03}).

Let $\cG$ be a compact $p$-adic Lie group.
We write $\dim \cG$ for the dimension of $\cG$ as a $p$-adic Lie group.
Let $\Lambda = \Z_p[[\cG]]$ be the completed group ring of $\cG$ over $\Z_p$.
We begin with listing basic facts (see \cite[\S 1.2]{Ven02} for a more detailed summary):

\begin{itemize}
\item
Any closed subgroup $\cG'$ of $\cG$ is again a $p$-adic Lie group (Cartan's theorem, see Serre \cite[Chap.~V, \S 9]{Ser06}).
\item
There is an open subgroup $\cG_0$ of $\cG$ that is pro-$p$ and $p$-torsion-free.
\item
The ring $\Lambda$ is (left and right) noetherian (Lazard \cite[Chap.~V, (2.2.4)]{Laz65}).
\end{itemize}

When $\cG$ is $p$-torsion-free, Venjakob \cite[Definition 3.1(iii)]{Ven02} introduced the concept of pseudo-nullity by showing that $\Lambda$ is an Auslander regular ring (\cite[Theorem 3.26]{Ven02}).
In any case, as in \cite{HS05}, let us define the pseudo-nullity in simple terms.

\begin{defn}\label{defn:PN}
A finitely generated $\Lambda$-module $M$ is pseudo-null if we have
\[
E^i_{\Lambda}(M) := \Ext_{\Lambda}^i(M, \Lambda) = 0
\]
for $i = 0 ,1$.
\end{defn}

\begin{rem}
It is easy to see that, when $\cG_0$ is an open subgroup of $\cG$ and we write $\Lambda_0 = \Z_p[[\cG_0]]$, we have an isomorphism 
\[
E_{\Lambda_0}^i(M) \simeq E_{\Lambda}^i(M)
\]
(see \cite[Proposition 2.7(ii)]{Ven02}).
In particular, the pseudo-nullity over $\Lambda$ and $\Lambda_0$ are equivalent.
Thus, to study the pseudo-nullity, we may assume $\cG$ is $p$-torsion-free (and pro-$p$ if necessary).

Let us check the equivalence of our definition and \cite[Definition 3.1(iii)]{Ven02} when $\cG$ is $p$-torsion-free.
As in \cite[Definition 3.3(i)]{Ven02}, we define the grade $j(M)$ of $M$ by 
\[
j(M) = \min \{i \mid E^i_{\Lambda}(M) \neq 0\}.
\]
Then our definition of the pseudo-nullity says $j(M) \geq 2$.
On the other hand, \cite[Definition 3.1(i)]{Ven02} introduces the dimension $\delta(M)$ of $M$ and define the pseudo-nullity as $\delta(M) \leq d - 2$, where $d = \dim \cG + 1$ is the homological dimension of $\Lambda$.
Now the equivalence follows from the formula $\delta(M) + j(M) = d$ (\cite[Proposition 3.5(ii)]{Ven02}).
\end{rem}

\begin{prop}\label{prop:PN_free}
Let $P$ be a finitely generated $\Lambda$-module such that $\pd_{\Lambda}(P) \leq 1$.
Then $P$ does not have nonzero pseudo-null submodules.
\end{prop}

\begin{proof}
This follows from \cite[Propositions 3.2(i) and 3.10]{Ven02}.
Let us give a direct proof by using Definition \ref{defn:PN}.
Let $0 \to F_1 \to F_0 \to P \to 0$ be a presentation of $P$ with $F_0, F_1$ finitely generated projective over $\Lambda$.
Let $M$ be a pseudo-null submodule of $P$.
By pull-back, we have a commutative diagram
\[
\xymatrix{
	0 \ar[r]
	& F_1 \ar[r]
	& F_0 \ar[r]
	& P \ar[r]
	& 0\\
	0 \ar[r]
	& F_1 \ar[r]_{f} \ar@{=}[u]
	& N \ar[r] \ar@{^(->}[u]
	& M \ar[r] \ar@{^(->}[u]
	& 0.
}
\]
By the pseudo-nullity of $M$, we have $\Ext_{\Lambda}^i(M, F_1) = 0$ for $i = 0, 1$, so the lower exact sequence induces an isomorphism
\[
f^*: \Hom_{\Lambda}(N, F_1) \to \Hom_{\Lambda}(F_1, F_1).
\]
Considering the lift of the identity in $\Hom_{\Lambda}(F_1, F_1)$ to $N$, we see that the injective map $f$ splits.
Therefore, the surjective map $N \to M$ also splits.
Since $N \subset F_0$ and $\Hom_{\Lambda}(M, F_0) = 0$, we must have $M = 0$.
\end{proof}

In \S \ref{ss:pf_main2}, we will investigate left $\Lambda = \Z_p[[\cG]]$-modules of the form
\[
\Z_p[[\cG/\cN]] = \Z_p[[\cG]] \otimes_{\Z_p[[\cN]]} \Z_p
\]
for closed subgroups $\cN$ of $\cG$.
We introduce two lemmas in advance.

\begin{lem}\label{lem:pdZ1}
For a closed subgroup $\cN$ of $\cG$, the following are equivalent.
\begin{itemize}
\item[(i)]
$\cN$ is $p$-torsion-free.
\item[(ii)]
We have $\pd_{\Lambda}(\Z_p[[\cG/\cN]]) = \dim \cN$.
\item[(iii)]
We have $\pd_{\Lambda}(\Z_p[[\cG/\cN]]) < \infty$.
\end{itemize}
\end{lem}

\begin{proof}
By \cite[Proposition 4.9 (iii)]{Ven02}, we have
\[
\pd_{\Lambda}(\Z_p[[\cG/\cN]])
= \pd_{\Z_p[[\cN]]}(\Z_p).
\]
Note that this reduces the proof to the case $\cN = \cG$.

By Brumer \cite[Corollary 4.4]{Bru66} (see \cite[page 275]{Ven02} or \cite[Corollary (5.2.13)]{NSW08}), we know that $\pd_{\Z_p[[\cN]]}(\Z_p)$ is equal to the $p$-cohomological dimension $\cd_p \cN$ of $\cN$.
If $\cN$ is not $p$-torsion-free, then $\cd_p \cN = \infty$ (e.g., \cite[(3.3.1)]{NSW08}).
If $\cN$ is $p$-torsion-free, since $\cN$ is a $p$-adic Lie group, it is known to be a Poincar\'{e} group of dimension $\dim \cN$ (Lazard \cite{Laz65}, see Serre \cite[I, \S 4.5]{Ser02}), so in particular we have $\cd_p \cN = \dim \cN$.
\end{proof}

\begin{lem}\label{lem:alg3}
Let $\cN$ be a closed subgroup of $\cG$ such that $\dim \cN \geq 2$.
Then $\Z_p[[\cG/\cN]]$ is pseudo-null as a $\Lambda$-module.
\end{lem}

\begin{proof}
It is easy to show
\[
E^i_{\Lambda}(\Z_p[[\cG/\cN]])
\simeq E^i_{\Z_p[[\cN]]}(\Z_p) \otimes_{\Z_p[[\cN]]} \Z_p[[\cG]]
\]
(see \cite[Proposition 2.7(i)]{Ven02}).
This reduces the proof to the case $\cN = \cG$.
Then the lemma follows from \cite[Corollary 4.8(i)]{Ven02}; indeed, it implies that the grade $j(\Z_p)$ over $\Z_p[[\cN]]$ equals $\dim \cN$.
\end{proof}

\section{Results for totally real fields}\label{sec:real}

In this section, we prove Theorem \ref{thm:main} and its variant (Theorem \ref{thm:main2}).
We keep the notation in \S \ref{ss:intro_real}.
Set $\Lambda = \Z_p[[\cG]]$.

\subsection{Notes on the statement}\label{ss:note}

First we review some properties of $\lambda$ and $\mu$-invariants.
Recall that $\lambda_S$ and $\mu_S$ are associated to the Iwasawa module $X_S(F^{\cyc})$.
It is well-known that the Iwasawa module $X_S(F^{\cyc})$ does not have nonzero finite submodules; this is a special case of Corollary \ref{cor:PN_free2} below.
Therefore, if $\mu_S = 0$, then $X_S(F^{\cyc})$ is a free $\Z_p$-module of rank $\lambda_S$.
In particular, we have $X_S(F^{\cyc}) = 0$ if and only if $\lambda_S = \mu_S = 0$.

The following basic lemma (cf.~\cite[Corollary (11.3.6)]{NSW08}) is unnecessary for the proof of the main theorems, but it enables us to reformulate Theorem \ref{thm:main}.

\begin{lem}\label{lem:S_varies}
We have $\mu_S = \mu_{S_p}$ and
\[
\lambda_S = \lambda_{S_p} + \# \{ v \in S(F^{\cyc}) \mid \mathsf{N}(v) \equiv 1 \, (\bmod p)\},
\]
where $S(F^{\cyc})$ denotes the set of primes of $F^{\cyc}$ that are lying above $S$ and $\mathsf{N}(v)$ denotes the cardinality of the residue field at $v$.
\end{lem}

\begin{proof}
We have a canonical surjective homomorphism from $X_S(F^{\cyc})$ to $X_{S_p}(F^{\cyc})$.
Its kernel is the direct sum of the projective limits of the $p$-parts of the local unit groups at primes in $S \setminus S_p$.
By computing the $\Z_p$-ranks, we obtain the lemma.
\end{proof}

Recall that $S_{\ram}^S = S_{\ram}^S(L/F^{\cyc})$ denotes the set of primes of $F^{\cyc}$ that are ramified in $L/F^{\cyc}$ and not lying above a prime in $S$.
We also write $(S_{\ram}^S)_L = S_{\ram}^S(L/F^{\cyc})_{L}$ for the set of primes of $L$ that are lying above a prime in $S_{\ram}^S$.

The following lemma is mentioned in \S \ref{ss:intro_real}.

\begin{lem}\label{lem:H_cyclic}
The following statements hold.
\begin{itemize}
\item[(1)]
Suppose $\mu_S = \lambda_S = 0$ and $\# S_{\ram}^S = 1$.
Then the unique prime $v_0 \in S_{\ram}^S$ is totally ramified in $L/F^{\cyc}$, so we have $\# (S_{\ram}^S)_L = 1$.
\item[(2)]
If $\mu_S = 0$ and $\lambda_S + \# S_{\ram}^S \leq 1$, then $\cH$ is pro-cyclic.
\end{itemize}
\end{lem}

\begin{proof}
(1)
Take a prime $w_0$ of $L$ lying above $v_0$ and let $I_{w_0}(L/F^{\cyc}) \subset \cH$ be its inertia group.
Then by $X_S(F^{\cyc}) = 0$, the natural map from $I_{w_0}(L/F^{\cyc})$ to the abelianization $\cH^{\ab}$ of $\cH$ is surjective.
This implies that $I_{w_0}(L/F^{\cyc}) = \cH$ (see, e.g., Serre \cite[I, \S 4.2, Proposition 23 bis]{Ser02}).

(2)
First, suppose $\# S_{\ram}^S = 1$.
Then by (1), $L/F^{\cyc}$ is totally ramified at the non-$p$-adic prime $v_0$, so it must be pro-cyclic.

Next, suppose $\# S_{\ram}^S = 0$.
Then since $L/F^{\cyc}$ is unramified outside $S$, the abelianization $\cH^{\ab}$ is a quotient of $X_S(F^{\cyc})$.
By the assumption, $X_S(F^{\cyc})$ is a free $\Z_p$-module of rank $\lambda_S \leq 1$, so $X_S(F^{\cyc})$ is pro-cyclic.
Therefore, $\cH^{\ab}$ is also pro-cyclic, which implies $\cH$ is already pro-cyclic (loc.~cit.).
\end{proof}

Here is a remark on the statement of Theorem \ref{thm:main} when $\cG$ is commutative.

\begin{rem}
Suppose that $\cG$ is commutative.
Also, suppose that $\cH$ is finite, as Leopoldt's conjecture predicts.
In this case, Theorem \ref{thm:main} claims that we have $X_S(L) = 0$ if and only if $X_S(F^{\cyc}) = 0$ and $\# S_{\ram}^S \leq 1$.

We can compare this with a result of Kurihara and the author \cite{KK_22} on the minimal number of generators of $X_S(L)$ over the Iwasawa algebra.
It deals with the case $S_{\ram}^S = \emptyset$ only.
In this case, \cite[Theorem 1.1]{KK_22} implies that we have $X_S(L) = 0$ if and only if $X_S(F^{\cyc}) = 0$ and $L = F^{\cyc}$.
This agrees with Theorem \ref{thm:main}.
\end{rem}

\subsection{Tate sequences}\label{ss:Tate_seq}

Let us introduce the exact sequences that play the crucial role in the proof of Theorem \ref{thm:main}.

For a finite prime $u$ of $F$, we choose a prime of $L$ lying above $u$ and write $\cG_u \subset \cG$ for the decomposition group.
Note that $\cG_u$ has an ambiguity up to inner automorphisms, but this does not matter in the subsequent argument.
We set
\[
Z_u(L) = \Z_p[[\cG/\cG_u]] = \Z_p[[\cG]] \otimes_{\Z_p[[\cG_u]]} \Z_p,
\]
which is regarded as a left module over $\Lambda = \Z_p[[\cG]]$.
Then for a finite set $T$ of finite primes of $F$, we set
\[
Z_T(L) = \bigoplus_{u \in T} Z_v(L).
\]
Moreover, when $T$ is non-empty, we define $Z_T^0(L)$ as the kernel of the natural surjective map $Z_T(L) \to \Z_p$, so we have an exact sequence
\begin{equation}\label{eq:Z0}
0 \to Z_T^0(L) \to Z_T(L) \to \Z_p \to 0.
\end{equation}

Let $S_{\ram}(L/F)$ denote the set of primes of $F$ that are ramified in $L/F$, which is assumed to be finite.
Let $S_{\infty} = S_{\infty}(F)$ denote the set of archimedean places of $F$.

Claim (1) in the next proposition is mentioned in \S \ref{ss:pf_outline}, while claim (2) is to handle the general case.
In fact, to prove the theorems in this paper, claim (1) is not necessary and (2) suffices in all cases.

\begin{prop}\label{prop:Tate_seq}
The following hold.
\begin{itemize}
\item[(1)]
Suppose $S \supset S_{\ram}(L/F)$.
Then we have an exact sequence
\[
0 \to X_S(L) \to P \to Q \to \Z_p \to 0
\]
over $\Lambda$, where $P$ and $Q$ are finitely generated $\Lambda$-modules with $\pd_{\Lambda} \leq 1$.
\item[(2)]
Let $\Sigma \supset S \cup S_{\ram}(L/F) \cup S_{\infty}$ be a finite set of places of $F$ with $\Sigma_f := \Sigma \setminus S_{\infty} \supsetneqq S$.
Then we have an exact sequence
\[
0 \to X_S(L) \to P \to Z_{\Sigma_f \setminus S}^0(L) \to 0
\]
over $\Lambda$, where $P$ is a finitely generated $\Lambda$-module with $\pd_{\Lambda} \leq 1$.
\end{itemize}
\end{prop}

\begin{proof}
At least when $\cG$ is commutative and $\cH$ is finite, claim (1) is well-known (e.g., \cite[Theorem 4.1]{KK_22} by Kurihara and the author) and claim (2) is a direct generalization of \cite[Proposition 2.11]{GKK22}.
Even in the non-commutative case, claim (1) follows from \cite[Proposition 2.13]{Ven03}.
Note that the weak Leopoldt's conjecture holds since $L$ contains $F^{\cyc}$.

Let us sketch the construction.
In case (1), set $\Sigma = S \cup S_{\infty}$ and $\Sigma_f = S$.
In both (1) and (2), let $C_S$ be a complex that is defined so that we have a triangle
\[
C_S \to \RG_{\Iw}(F_{\Sigma}/L, \Z_p(1)) \to \bigoplus_{v \in S} \RG_{\Iw}(L_v, \Z_p(1)),
\]
where $F_{\Sigma}$ denotes the maximal extension of $F$ unramified outside $\Sigma$ and $\RG_{\Iw}$ denotes the (global and local) Iwasawa cohomology complexes (we follow the notation in the author's article \cite[\S 3]{Kata_12}).
Since $\Sigma \supset S_{\ram}(L/F)$, it is known that both $\RG_{\Iw}$ are perfect (Fukaya--Kato \cite[Proposition 1.6.5]{FK06}), so $C_S$ is also perfect.
The global duality gives us a triangle
\[
 \RG_{\Iw}(F_{\Sigma}/L, \Z_p(1)) 
 \to \bigoplus_{v \in \Sigma_f} \RG_{\Iw}(L_v, \Z_p(1))
\to \RG(F_{\Sigma}/L, \Q_p/\Z_p)^{\vee}[-2],
\]
where $(-)^{\vee}$ denotes the Pontryagin dual
(see Nekov\'{a}\v{r} \cite[\S 5.4]{Nek06}; in the book the coefficient ring is assumed to be commutative, but the proof is valid for our non-commutative case).
Therefore, we obtain a triangle
\begin{equation}\label{eq:PD}
C_S \to \bigoplus_{v \in \Sigma_f \setminus S} \RG_{\Iw}(L_v, \Z_p(1)) 
\to \RG(F_{\Sigma}/L, \Q_p/\Z_p)^{\vee}[-2].
\end{equation}

In case (1), triangle \eqref{eq:PD} means $C_S \simeq \RG(F_{\Sigma}/L, \Q_p/\Z_p)^{\vee}[-3]$, so $H^2(C_S) \simeq X_S(L)$, $H^3(C_S) \simeq \Z_p$, and $H^i(C_S) = 0$ for $i \neq 2, 3$.
By taking a quasi-isomorphism $C_S \simeq [0 \to C^1 \to C^2 \to C^3 \to 0]$ with $C^1, C^2, C^3$ projective, we obtain a sequence of the desired form (cf.~\cite[Theorem 4.1]{KK_22}).

In case (2), the cohomology exact sequence of \eqref{eq:PD} yields
\[
0 \to X_S(L) \to H^2(C_S) \to Z_{\Sigma_f \setminus S}(L) \to \Z_p \to 0
\]
and $H^i(C_S) = 0$ for $i \neq 2$.
By taking a quasi-isomorphism $C_S \simeq [0 \to C^1 \to C^2 \to 0]$ with $C^1, C^2$ projective, we see that $H^2(C_S)$ satisfies $\pd_{\Lambda} \leq 1$.
By setting $P = H^2(C_S)$, we obtain a sequence of the desired form.
\end{proof}

The $S \supset S_{\ram}(L/F)$ case of the following corollary is proved in \cite[Theorem 4.5]{Ven03}.

\begin{cor}\label{cor:PN_free2}
The $\Lambda$-module $X_S(L)$ does not have nonzero pseudo-null submodules.
\end{cor}

\begin{proof}
This immediately follows from Propositions \ref{prop:PN_free} and \ref{prop:Tate_seq}(2).
\end{proof}

In particular, $X_S(L)$ is pseudo-null if and only if $X_S(L) = 0$.
In what follows, we freely use this fact.

\subsection{Rephrasing the theorem}\label{ss:rephrase}

To prove Theorem \ref{thm:main}, it is useful to rephrase it in the following form.

\begin{thm}\label{thm:main'}
We have $X_S(L) = 0$ if and only if (exactly) one of (A), (B), and (C) holds:
\begin{itemize}

\item[(A)]
Both (A1) and (A2) hold.
\begin{itemize}
\item[(A1)]
$L = F^{\cyc}$.
\item[(A2)]
$\mu_S = 0$ and $\lambda_S = 0$.
\end{itemize}

\item[(B)]
Both (B1) and (B2) hold.
\begin{itemize}
\item[(B1)]
$S_{\ram}^S = \emptyset$ and $\cH \simeq \Z_p$.
\item[(B2)]
$\mu_S = 0$ and $\lambda_S = 1$.
\end{itemize}

\item[(C)]
Both (C1) and (C2) hold.
\begin{itemize}
\item[(C1)]
$\# (S_{\ram}^S)_L = 1$.
\item[(C2)]
$\mu_S = 0$ and $\lambda_S = 0$.
\end{itemize}

\end{itemize}
\end{thm}

Let us briefly discuss the equivalence between Theorems \ref{thm:main} and \ref{thm:main'}.
The $\mu_S = 0$ does not matter at all.
We easily see that ``(A) $\Rightarrow$ (i)'' and ``(B) $\Rightarrow$ (ii).''
We also have ``(C) $\Rightarrow$ $\lambda_S = 0$ and $\# S_{\ram}^S = 1$ $\Rightarrow$ (i) or (ii),'' where the final implication holds since $\dim \cG = 2$ in (ii) can be replaced by $\dim \cG \geq 2$ by Lemma \ref{lem:H_cyclic}(2).
Next, we easily see that ``(i) and $\# S_{\ram}^S = 0$ $\Rightarrow$ (A)'' and ``(ii) and $\# S_{\ram}^S = 0$ $\Rightarrow$ (B).''
If (i) or (ii) holds and moreover $\# S_{\ram}^S = 1$, we have (C) by Lemma \ref{lem:H_cyclic}(1).

\begin{eg}\label{eg:Q}
When $F = \Q$, let us explicitly describe the situations (A), (B), and (C).
It is well-known that $\mu_{S_p} = 0$ and $\lambda_{S_p} = 0$.
Therefore, Lemma \ref{lem:S_varies} tells us $\mu_S = 0$ and what $\lambda_S$ is in general.
Without loss of generality, we assume that every $\ell \in S \setminus \{p\}$ satisfies $\ell \equiv 1  \, (\bmod p)$.
Then conditions (A), (B), and (C) are described as follows:
\begin{itemize}
\item[(A)]
$S = \{p\}$ and $L = \Q^{\cyc}$.

\item[(B)]
$S = \{p, \ell\}$ with $\ell$ non-split in $\Q^{\cyc}/\Q$, and $L = M_S(\Q^{\cyc})$.

Note that $M_S(\Q^{\cyc})$ is a $\Z_p$-extension of $\Q^{\cyc}$.

\item[(C)]
$S = \{ p \}$ and $L/\Q^{\cyc}$ is ramified at a unique non-$p$-adic prime of $L$ (or of $\Q^{\cyc}$).

For instance, this occurs if we choose a prime number $\ell$ as in (B) and take $L$ as any intermediate field of $M_{\{p, \ell\}}(\Q^{\cyc})/\Q^{\cyc}$ other than $\Q^{\cyc}$.
\end{itemize}

By Theorem \ref{thm:main'}, we have $X_S(L) \neq 0$ except for these special cases.
\end{eg}

To prove Theorem \ref{thm:main'}, we begin with observing the Galois coinvariant module $X_S(L)_{\cH}$.

\begin{lem}\label{lem:H_coinv}
If (B1) holds, then $X_S(L)_{\cH}$ is isomorphic to the kernel of the natural homomorphism $X_S(F^{\cyc}) \twoheadrightarrow \cH$.
If (C1) holds, then we have an isomorphism $X_S(L)_{\cH} \simeq X_S(F^{\cyc})$.
\end{lem}

\begin{proof}
This lemma can be deduced from the Tate sequences in Proposition \ref{prop:Tate_seq} and their functoriality, but we give a direct proof here.

In any case $\cH$ is pro-cyclic, so we have
\[
X_S(L)_{\cH} \simeq \Gal(\cM/L)
\]
if $\cM$ denotes the maximal abelian extension of $F^{\cyc}$ in $M_S(L)$.
If (B1) holds, then $\cM = M_S(F^{\cyc})$, so the claim follows.

Suppose (C1) holds and let $v_0$ be the unique prime of $F^{\cyc}$ that is ramified in $L/F^{\cyc}$.
Then $M_S(F^{\cyc})$ is the fixed subfield in $\cM$ of the inertia subgroup $I_{v_0}(\cM/F^{\cyc})$.
Since we have $I_{v_0}(\cM/F^{\cyc}) \simeq I_{v_0}(L/F^{\cyc}) = \Gal(L/F^{\cyc})$, the claim follows.
\end{proof}

By Lemma \ref{lem:H_coinv}, together with Nakayama's lemma, we immediately obtain the following.

\begin{lem}
Assuming (A1) (resp.~(B1), resp.~(C1)), we have $X_S(L) = 0$ if and only if (A2)  (resp.~(B2), resp.~(C2)) holds.
\end{lem}

Thanks to this lemma, in order to prove the full statement of Theorem \ref{thm:main'}, it is enough to show that ``$X_S(L) = 0$ $\Rightarrow$ (A1), (B1), or (C1).''
To show this implication, assuming $X_S(L) = 0$, we observe that Proposition \ref{prop:Tate_seq}(2) shows
\begin{equation}\label{eq:pd_Z}
\pd_{\Lambda} (Z_{\Sigma_f \setminus S}^0(L)) \leq 1.
\end{equation}
In fact, this property is all we need:

\begin{prop}\label{prop:key}
If $\pd_{\Lambda} (Z_{\Sigma_f \setminus S}^0(L)) \leq 1$, then one of (A1), (B1), and (C1) holds.
\end{prop}

This proposition will be shown in the next subsection.

\subsection{Homological argument}\label{ss:pf_main2}

We consider the following abstract situation.

Let $\cG$ be a pro-$p$, $p$-adic Lie group.
Let $\cH \subset \cG$ be a normal closed subgroup such that $\cG/\cH$ is isomorphic to $\Z_p$.
Let $\{ \cG_j\}_{j \in J}$ be a family of closed subgroups of $\cG$ with $J$ a non-empty finite set.
We suppose $\cG_j \not\subset \cH$ for any $j \in J$; note that this implies $\dim \cG_j \geq 1$ since $\cG/\cH \simeq \Z_p$.
Set $Z_j = \Z_p[[\cG/\cG_j]]$ and $Z_J = \bigoplus_{j \in J} Z_j$.
Also, let $Z_J^0$ be the kernel of the natural homomorphism $Z_J \to \Z_p$, so we have an exact sequence
\begin{equation}\label{eq:Z^0_J}
0 \to Z^0_J \to Z_J \to \Z_p \to 0.
\end{equation}
For each $j \in J$, set $\cH_j = \cH \cap \cG_j$.

The following is the abstract version of Proposition \ref{prop:key}.

\begin{prop}\label{prop:key_abs}
If $\pd_{\Z_p[[\cG]]} (Z^0_J) \leq 1$, then one of the following holds:
\begin{itemize}
\item[(a)]
$\cH$ is trivial.
\item[(b)]
$\cH_j$ is trivial for all $j \in J$, and $\cH$ is isomorphic to $\Z_p$.
\item[(c)]
$\cH$ is non-trivial, there is $j_0 \in J$ such that $\cG_{j_0} = \cG$, and $\cH_j$ is trivial for any $j \in J \setminus \{j_0\}$.
\end{itemize}
\end{prop}

This proposition is an immediate consequence of Claims \ref{claim:4}--\ref{claim:5} below.

\begin{claim}\label{claim:4}
Suppose that $\dim \cG_{j_0} \geq 2$ for some $j_0 \in J$.
Then (c) holds.
\end{claim}

\begin{proof}
By $\pd_{\Z_p[[\cG]]} (Z^0_J) \leq 1$ and Proposition \ref{prop:PN_free}, the module $Z^0_J$ does not have nonzero pseudo-null submodules.
On the other hand, $Z_{j_0}$ is pseudo-null by the assumption and Lemma \ref{lem:alg3}.
Therefore, the homomorphism $Z_{j_0} \to \Z_p$ must be injective.
This means that $\cG_{j_0} = \cG$ and then we obtain
\[
Z^0_J \simeq Z_{J \setminus \{j_0\}} := \bigoplus_{j \in J \setminus \{j_0\}} Z_j.
\]
So $\pd_{\Z_p[[\cG]]} (Z^0_J) \leq 1$ implies $\pd_{\Z_p[[\cG]]}(Z_j) \leq 1$ for any $j \in J \setminus \{j_0\}$.
By Lemma \ref{lem:pdZ1}, this means $\cG_j \simeq \Z_p$, that is, $\cH_j$ is trivial.
Therefore, (c) holds.
\end{proof}

\begin{claim}\label{claim:1}
Suppose that $\dim \cG_j = 1$ for all $j \in J$.
Then we have $\dim \cG \leq 2$.
\end{claim}

\begin{proof}
We take any open subgroup $\cG_0 \subset \cG$ without $p$-torsion.
By Lemma \ref{lem:pdZ1}, we have $\pd_{\Z_p[[\cG_0]]}(\Z_p) = \dim \cG$ and $\pd_{\Z_p[[\cG_0]]}(Z_j) = \dim \cG_j = 1$ for any $j \in J$.
Then \eqref{eq:Z^0_J} and $\pd_{\Z_p[[\cG_0]]} (Z^0_J) \leq 1$ show the claim.
\end{proof}

\begin{claim}\label{claim:3}
Suppose that $\dim \cG = 1$.
Then (a) or (c) holds.
\end{claim}

\begin{proof}
Let us assume $\cH$ is non-trivial and finite, so the claim says (c) holds.
By \eqref{eq:Z^0_J} and $\pd_{\Z_p[[\cG]]} (Z^0_J) \leq 1 < \infty$, we have an isomorphism between Tate cohomology groups
\[
\hat{H}^0(\cH, Z_J) \simeq \hat{H}^0(\cH, \Z_p).
\]
For each $j \in J$, we have
\begin{align}
\hat{H}^0(\cH, Z_j) 
& \simeq \hat{H}^0(\cH, \Z_p[\cG/\cG_j])\\
& \simeq \Z_p[\cG] \otimes_{\Z_p[\cG_j \cH]} \hat{H}^0(\cH, \Z_p[\cG_j \cH/\cG_j])\\
& \simeq \Z_p[\cG] \otimes_{\Z_p[\cG_j \cH]} \hat{H}^0(\cH, \Z_p[\cH/\cH_j])\\
& \simeq \Z_p[\cG] \otimes_{\Z_p[\cG_j \cH]} \hat{H}^0(\cH_j, \Z_p).
\end{align}
Thus, we obtain
\[
\bigoplus_{j \in J} \Z_p[\cG/\cG_j \cH]/(\# \cH_j) \simeq \Z_p/(\# \cH).
\]
Therefore, there exists a single $j_0 \in J$ such that $\cH_{j_0} = \cH$ and $\cG_{j_0} \cH = \cG$, and $\cH_j$ is trivial for any other $j$.
Then we indeed have $\cG_{j_0} = \cG$.
Therefore, (c) holds.
\end{proof}

\begin{claim}\label{claim:5}
Suppose that $\dim \cG = 2$ and $\dim \cG_j = 1$ for any $j \in J $.
Then (b) holds.
\end{claim}

\begin{proof}
It is enough to show that $\cH$ is isomorphic to $\Z_p$, since then $\cG$ is $p$-torsion-free and so $\cH_j$ must be trivial for any $j \in J$.
Let $\cN$ be any open subgroup of $\cH$ such that $\cN$ is normal in $\cG$ and $\cN \simeq \Z_p$.
We shall show that the group $\Delta := \cH/\cN$ is cyclic.
Then we would deduce that $\cH$ is a pro-cyclic group (take the limit with respect to $\cN$, or consider a single $\cN$ that is contained in the Frattini subgroup of $\cH$), which shows $\cH \simeq \Z_p$, as desired.

By the assumption $\dim \cG_j = 1$, we see that $Z_j$ is finitely generated and free as a $\Z_p[[\cN]]$-module.
Also, we clearly have 
\[
(Z_J)_\cN
\simeq 
Z_{J, \cN} := \bigoplus_{j \in J} \Z_p[\cG/\cN \cG_j].
\]
Therefore, taking the $\cN$-homology of \eqref{eq:Z^0_J}, we obtain an exact sequence
\[
0 \to H_1(\cN, \Z_p) \to (Z^0_J)_{\cN} \to Z_{J, \cN} \to \Z_p \to 0
\]
over $\Z_p[[\cG/\cN]]$.
We have $H_1(\cN, \Z_p) \simeq \cN^{\ab} = \cN$.
We observe that the action of $\Delta$ on $\cN$ is trivial, simply because the automorphic group of $\cN$ is $\Z_p^{\times}$, which does not have non-trivial finite $p$-group (since $p \geq 3$).
By $\pd_{\Z_p[[\cG]]} (Z^0_J) \leq 1$, we have $\pd_{\Z_p[[\cG/\cN]]}((Z^0_J)_{\cN}) \leq 1$ as well.
Then, taking the $\Delta$-cohomology, we have a long exact sequence
\[
\cdots \to \hat{H}^{i+1}(\Delta, \cN)
\to \hat{H}^i(\Delta,  Z_{J, \cN})
\to \hat{H}^{i}(\Delta, \Z_p)
\to \cdots
\]
In particular, we obtain
\[
\hat{H}^{-2}(\Delta, \Z_p)
\to \hat{H}^{0}(\Delta, \cN)
\to \hat{H}^{-1}(\Delta, Z_{J, \cN}).
\]
This is identified with $\Delta^{\ab} \to \Z_p/(\# \Delta) \to 0$.
Therefore, we conclude that $\Delta$ is a cyclic group, as claimed.
\end{proof}

This completes the proof of Proposition \ref{prop:key_abs}.
This also completes the proof of Theorem \ref{thm:main'} and, equivalently, of Theorem \ref{thm:main}.

\subsection{A variant}\label{ss:variant}

In this subsection, we obtain a variant of Theorem \ref{thm:main} that is rather easier.

Suppose we are given a totally real field $F'$ that is a finite abelian extension of $F$ whose degree is prime to $p$.
Set $L' = F' L$, so we have a natural isomorphism $\Gal(L'/L) \simeq \Gal(F'/F)$, which we write $\Delta$.
Let $\psi$ be a (totally even) character of $\Delta$.
Set $\cO_{\psi} = \Z_p[\Imag(\psi)]$ and we regard this as a $\Z_p[\Delta]$-algebra via $\psi$.
For a $\Z_p[[\Gal(L'/F)]]$-module $M$, we consider the $\cO_{\psi}[[\cG]]$-module
\[
M^{\psi} := \cO_{\psi} \otimes_{\Z_p[\Delta]} M.
\]
This functor $(-)^{\psi}$ is exact since the order of $\Delta$ is prime to $p$.
When $\psi$ is the trivial character, we may identify $X_S(L')^{\psi}$ with $X_S(L)$, which we studied in Theorem \ref{thm:main}.

To state the result, we introduce the following:
\begin{itemize}
\item
Let $S_{\ram}^{S, \psi} = S_{\ram}^{S, \psi}(L'/F^{\cyc})$ be the set of primes of $F^{\cyc}$ that are ramified in $L/F^{\cyc}$, totally split in $F^{\cyc, \psi}/F^{\cyc}$, and not lying above $S$.
Here, $F^{\cyc, \psi}$ denotes the extension of $F^{\cyc}$ cut out by $\psi$, that is, the extension corresponding to the kernel of $\psi$.
\item
Let $\lambda_S^{\psi}$, $\mu_S^{\psi}$ be the Iwasawa $\lambda$, $\mu$-invariants of $X_S(F^{\prime, \cyc})^{\psi}$.
\end{itemize}

\begin{thm}\label{thm:main2}
Suppose $\psi$ is non-trivial.
Then the $\cO_{\psi}[[\cG]]$-module $X_S(L')^{\psi}$ is pseudo-null if and only if $\lambda_S^{\psi} = 0$, $\mu_S^{\psi} = 0$, and $\# S_{\ram}^{S, \psi} = 0$.
\end{thm}

Note that unlike Theorem \ref{thm:main}, we do not have an upper bound on $\dim \cG$.

\begin{proof}
We use the exact sequence in Proposition \ref{prop:Tate_seq}(2) applied to $L'/F$ instead of $L/F$.
Then we still see that $X_S(L')^{\psi}$ does not have nonzero pseudo-null submodules.
Therefore, the pseudo-nullity of $X_S(L')^{\psi}$ is equivalent to its vanishing.

An easy step is to observe that, assuming $\# S_{\ram}^{S, \psi} = 0$, we have $X_S(L')^{\psi} = 0$ if and only if $X_S(F^{\prime, \cyc})^{\psi} = 0$.
To show this, we only have to observe $(X_S(L')^{\psi})_{\cH} \simeq X_S(F^{\prime, \cyc})^{\psi}$ and use Nakayama's lemma.

Now we only have to show $\# S_{\ram}^{S, \psi} = 0$, assuming $X_S(L')^{\psi} = 0$.
By $X_S(L')^{\psi} = 0$ and the non-triviality of $\psi$, Proposition \ref{prop:Tate_seq}(2) shows
\[
\pd_{\cO_{\psi}[[\cG]]}(Z_u(L')^{\psi}) \leq 1
\]
for any $u \not \in S$.
This is a counterpart of \eqref{eq:pd_Z}.
We have $Z_u(L')^{\psi} \neq 0$ if and only if $\psi$ is trivial on $\cG_u$.
In this case, by (the obvious variant of) Lemma \ref{lem:pdZ1}, the property $\pd_{\cO_{\psi}[[\cG]]}(Z_u(L')^{\psi}) \leq 1$ means $\cG_u \simeq \Z_p$, that is, $u$ is unramified in $L/F^{\cyc}$.
Thus, $\# S_{\ram}^{S, \psi} = 0$ holds.
\end{proof}

\section{Results for CM-fields}\label{ss:duality}

We keep the notation in \S \ref{ss:intro_CM}.
In \S \ref{ss:pf_CM}, we prove Theorem \ref{thm:main_CM} by using Theorems \ref{thm:main} and \ref{thm:main2}.
In \S \ref{ss:HS}, we compare Theorem \ref{thm:main_CM} with the work of Hachimori--Sharifi \cite{HS05}.

\subsection{Proof of Theorem \ref{thm:main_CM}}\label{ss:pf_CM}

We set $L' = \cL(\mu_p)^+$ and $F' = \cF(\mu_p)^+$.
Let $\chi$ be the non-trivial character of $\Gal(\cF/F)$.
Let $\omega$ be the Teichm\"{u}ller character of $\Gal(\cF(\mu_p)/F)$.
Set $\psi = \omega \chi^{-1}$.
Then $\psi$ is a totally real character of $\Gal(\cF(\mu_p)/F)$, so it factors through $\Gal(F'/F)$.
Note that $\psi$ is trivial if and only if $\cF = F(\mu_p)$, that is, $\delta = 1$.

\begin{prop}\label{prop:lm_dual}
The following statements hold.
\begin{itemize}
\item[(1)]
We have $\lambda^- = \lambda_{S_p}^{\psi}$ and $\mu^- = \mu_{S_p}^{\psi}$, where $\lambda_{S_p}^{\psi}$ and $\mu_{S_p}^{\psi}$ are as in \S \ref{ss:variant}.
\item[(2)]
The $\Z_p[[\cG]]$-module $X(\cL)^-$ is pseudo-null if and only if the $\Z_p[[\cG]]$-module $X_{S_p}(L')^{\psi}$ is pseudo-null.
\end{itemize}
\end{prop}

\begin{proof}
Claim (1) is well-known (see, e.g., \cite[Corollary (11.4.4)]{NSW08}).
Even claim (2) might be known to experts, but the author has not found a reference (a close one is \cite[Theorem 4.9]{Ven03}), so we include the proof here.

As in the proof of Proposition \ref{prop:Tate_seq}, we consider the Selmer complex $\cC_{S_p}$ over the Iwasawa algebra $\wtil{\Lambda} := \Z_p[[\Gal(\cL'/F)]]$, defined by a triangle
\[
\cC_{S_p} 
\to \RG_{\Iw}(\cL_{\Sigma}/\cL', \Z_p(1)) 
\to \bigoplus_{v \in S_p} \RG_{\Iw}(\cL'_v, \Z_p(1))
\]
with $\Sigma = S_p \cup S_{\infty}$.
We do not care whether this complex is perfect or not.
Set $\cC = \RG_{\Iw}(\cL_{\Sigma}/\cL', \Z_p(1))$.
Then the Artin--Verdier duality implies $\cC_{S_p} \simeq \cC(-1)^{*, \iota}[-3]$, where $(-)^*$ denotes the linear dual $\RG_{\wtil{\Lambda}}(-, \wtil{\Lambda})$ and $(-)^{\iota}$ denotes the involution that inverts the group elements of $\wtil{\Lambda}$ (see Nekov\'{a}\v{r} \cite[\S 5.4]{Nek06} or \cite[\S 3]{Kata_12}; the proof is again valid for our non-commutative case).
In particular, by taking the $\psi$-components, we obtain 
\begin{equation}\label{eq:dual}
\cC_{S_p}^{\psi} \simeq \cC^{\chi}(-1)^{*, \iota}[-3].
\end{equation}
As in the proof of Proposition \ref{prop:Tate_seq}(1), we have $H^i(\cC_{S_p}^{\psi}) = 0$ for $i \neq 2, 3$ and
\[
H^2(\cC_{S_p}^{\psi}) \simeq X_{S_p}(L')^{\psi}, 
\quad
H^3(\cC_{S_p}^{\psi}) \simeq (\Z_p)^{\psi} \simeq
\begin{cases}
	\Z_p & \text{if $\delta = 1$}\\
	0 & \text{if $\delta = 0$}.
\end{cases}
\]
In a similar way, we have $H^i(\cC^{\chi}) = 0$ for $i \neq 1, 2$, 
\[
H^1(\cC^{\chi}) \simeq 
\begin{cases}
	\Z_p(1) & \text{if $\delta = 1$}\\
	0 & \text{if $\delta = 0$},
\end{cases}
\]
and an exact sequence
\[
0 \to X'(\cL)^- \to H^2(\cC^{\chi}) \to Z_{S_p}^0(\cL)^- \to 0,
\]
where $X'(\cL)$ denotes the split Iwasawa module for $\cL$, which is defined as the quotient of $X(\cL)$ by requiring that all $p$-adic primes split completely.
By \cite[Lemma (11.4.9)]{NSW08}, we know an exact sequence
$
0 \to Z_{S_p}(\cL)^- \to X(\cL)^- \to X'(\cL)^- \to 0.
$

Note that, by applying the above argument to $\cL = \cF^{\cyc}$, we obtain
\[
\lambda(H^2(\cC_{S_p}^{\psi})) = \lambda_{S_p}^{\psi},
\quad
\mu(H^2(\cC_{S_p}^{\psi})) = \mu_{S_p}^{\psi}
\]
and
\[
\lambda(H^2(\cC^{\chi})) = \lambda^-,
\quad
\mu(H^2(\cC^{\chi})) = \mu^-.
\]
Also, we have $\lambda(H^3(\cC_{S_p}^{\psi})) = \lambda(H^1(\cC^{\chi})) = \delta$ and $\mu(H^3(\cC_{S_p}^{\psi})) = \mu(H^1(\cC^{\chi})) = 0$.
Thus, claim (1) follows from \eqref{eq:dual}.

Let us show claim (2).
We may assume that $\cG$ is $p$-torsion-free so that the homological dimension of $\Lambda$ is finite.
When $\dim \cG = 1$, the claim follows at once from (1).
Let us assume $\dim \cG \geq 2$, so both $H^3(\cC_{S_p}^{\psi})$ and $H^1(\cC^{\chi})$ are pseudo-null.
Also, $X(\cL)^-$ is pseudo-null if and only if so is $H^2(\cC^{\chi})$.
Therefore, by \eqref{eq:dual}, we only have to show that for a perfect complex $C$ whose cohomology groups are all pseudo-null, the same is true for $C^*$.
This claim follows from the Auslander regularity (\cite[Definition 3.3(ii)]{Ven02}) of $\Lambda$ and the fact that pseudo-nullity is closed under taking submodules, quotients, and extensions (see \cite[Proposition 3.6(ii)]{Ven02}).
\end{proof}

\begin{prop}\label{prop:Sram_dual}
We have $S_{\ram}^-(\cL/F^{\cyc}) = S_{\ram}^{S_p, \psi}(L'/F^{\cyc})$, where $S_{\ram}^{S_p, \psi}(L'/F^{\cyc})$ is as in \S \ref{ss:variant}.
\end{prop}

\begin{proof}
Let $v$ be a non-$p$-adic prime of $F^{\cyc}$ that is ramified in $L/F^{\cyc}$.
We have to show that $v$ is split in $\cF^{\cyc}/F^{\cyc}$ if and only if $v$ is totally split in $F^{\cyc, \psi}/F^{\cyc}$.

Since $v$ is ramified in a $p$-extension, by local class field theory, the completion of $F^{\cyc}$ at $v$ contains $\mu_p$.
In other words, $v$ is totally split in $F^{\cyc}(\mu_p)/F^{\cyc}$, which is the extension cut out by $\omega$.
On the other hand, the extension of $F^{\cyc}$ cut out by $\psi$ and $\chi$ are $F^{\cyc, \psi}$ and $\cF^{\cyc}$ respectively.
Therefore, the claim follows.
\end{proof}

Now Theorem \ref{thm:main_CM} follows immediately from Theorems \ref{thm:main} and \ref{thm:main2} applied to $S = S_p$, taking Propositions \ref{prop:lm_dual} and \ref{prop:Sram_dual} into account.

\subsection{The work of Hachimori--Sharifi}\label{ss:HS}

We still keep the notation in \S \ref{ss:intro_CM}.

\begin{thm}[{Hachimori--Sharifi \cite[Theorem 1.2]{HS05}}]\label{thm:HS05}
Suppose that the extension $\cL/\cF$ is strongly admissible, which means that $\dim \cG \geq 2$ and $\cG$ is $p$-torsion-free.
Also, we assume $\mu^- = 0$.
Then 
\[
\rank_{\Z_p[[\cH]]}(X(\cL)^-) = \lambda^- - \delta + \# S_{\ram}^-.
\]
In particular, $X(\cL)^-$ is pseudo-null over $\Z_p[[\cG]]$ if and only if
\[
\lambda^- + \# S_{\ram}^- = \delta.
\]
\end{thm}

It is straightforward to deduce the final part of this theorem from Theorem \ref{thm:main_CM}.
In fact, Theorem \ref{thm:main_CM} says much more: 
We have removed the ``strongly admissibility'' of $\cL/\cF$ and, moreover, we have shown that $\mu^- = 0$ follows from the pseudo-nullity of $X(\cL)^-$.
This answers a question in \cite[Example 5.3]{HS05} affirmatively.

The original proof of Theorem \ref{thm:HS05} relies on Kida's formula, which describes the relation between the Iwasawa invariants of $X(\cF^{\cyc})^-$ and $X(\cL)^-$ when $\cL/\cF^{\cyc}$ is a finite extension.
For strongly admissible $\cL/\cF$, by applying Kida's formula for the finite subextensions of $\cL/\cF^{\cyc}$, they succeeded in determining the $\Z_p[[\cH]]$-rank of $X(\cL)^-$.
The method in this paper cannot be used to recover this quantitative result, but the author \cite{Kata_27} recently obtained a more direct proof of Theorem \ref{thm:HS05} from the perspective of Selmer complexes.

\section*{Acknowledgments}

This work is supported by JSPS KAKENHI Grant Number 22K13898.

{
\bibliographystyle{abbrv}
\bibliography{biblio}
}

\end{document}